\documentclass[a4paper,12pt,intlimits,oneside]{amsart}

\usepackage{enumerate}
\usepackage{amsfonts,amsmath}

\usepackage{latexsym,amssymb}

\newcommand{\comment}[1]{}

\newcommand{\bR}{{\mathbb R}}

\newcommand{\bZ}{{\mathbb Z}}

\newcommand{\R}{{\mathbb R}}

\def\C{{\mathcal C}}
\def\H{{\mathcal H}}

\def\F{{\mathcal F}}

\def\S{{\mathcal S}}

\def\BMO{B\! M\! O}

\def\div{{\mbox{\small\rm  div}}\,}
\def\curl{{\mbox{\small\rm  curl}}\,}
\def\dist {{\mbox{\small\rm  dist}}\,}

\newcounter{rea}
\newcounter{rek}
\newcounter{res}
\begin{document}

\title[Paraproducts and Products through wavelets]{Paraproducts and Products of functions in $BMO(\mathbb R^n)$ and $H^1(\mathbb R^n)$ through wavelets}         
\author{ Aline {\bf BONAMI}}
\address{MAPMO-UMR 6628,
D\'epartement de Math\'ematiques, Universit\'e d'Orleans, 45067
Orl\'eans Cedex 2, France} 
\email{{\tt
Aline.Bonami@univ-orleans.fr}}

\author{Sandrine {\bf GRELLIER }}
\address{MAPMO-UMR 6628,
D\'epartement de Math\'ematiques, Universit\'e d'Orleans, 45067
Orl\'eans Cedex 2, France} 
\email{{\tt
Sandrine.Grellier@univ-orleans.fr}}
\author{Luong Dang {\bf KY}}    
\address{MAPMO-UMR 6628,
D\'epartement de Math\'ematiques, Universit\'e d'Orleans, 45067
Orl\'eans Cedex 2, France} 
\email{{\tt dangky@math.cnrs.fr}}

\begin{abstract} In this paper, we prove that the  product (in the
distribution sense) of two functions, which are respectively in $
\BMO(\bR^n)$ and $\H^1(\bR^n)$,
may be written as the sum of two continuous bilinear operators, one
from $\H^1(\bR^n)\times \BMO(\bR^n) $ into $L^1(\bR^n)$, the other one
from $\H^1(\bR^n)\times \BMO(\bR^n) $ into a new kind of Hardy-Orlicz
space denoted by $\H^{\log}(\bR^n)$. More precisely, the space
$\H^{\log}(\bR^n)$ is the set of distributions $f$ whose grand maximal
function $\mathcal Mf$ satisfies
$$\int_{\mathbb R^n} \frac {|\mathcal M f(x)|}{ \log(e+|x|) +\log (e+
|\mathcal Mf(x)|)}dx <\infty.$$
The two bilinear operators can be defined in terms of paraproducts. As
a consequence, we find an endpoint estimate involving the space
$\H^{\log}(\bR^n)$ for the $\div$-$\curl$ lemma.
\end{abstract}
\maketitle
\newtheorem{theorem}{Theorem}[section]
\newtheorem{lemma}{Lemma}[section]
\newtheorem{proposition}{Proposition}[section]
\newtheorem{remark}{Remark}[section]
\newtheorem{corollary}{Corollary}[section]
\newtheorem{definition}{Definition}[section]
\newtheorem{example}{Example}[section]
\numberwithin{equation}{section}
\newtheorem{Theorem}{Theorem}[section]
\newtheorem{Lemma}{Lemma}[section]
\newtheorem{Proposition}{Proposition}[section]
\newtheorem{Remark}{Remark}[section]
\newtheorem{Corollary}{Corollary}[section]
\newtheorem{Definition}{Definition}[section]
\newtheorem{Example}{Example}[section]

\section{Introduction}

Products of functions in $\H^1$ and $\BMO$ have been considered by Bonami, Iwaniec, Jones and Zinsmeister in
\cite{BIJZ}. Such products make sense as distributions, and can be written as the sum of an integrable function and a function in a weighted Hardy-Orlicz space. To be more precise, for $f\in\H^1(\bR^n)$ and $g\in \BMO(\bR^n)$, we define the product (in the distribution
 sense)
$f g$ as the distribution whose action on the Schwartz function
$\varphi\in\mathcal S(\bR^{n})$ is given by
\begin{equation}
\left\langle f g,\varphi\right\rangle:=\left\langle \varphi
g,f\right\rangle,
\end{equation}
where the second bracket stands for the duality bracket between
$\H^1(\bR^{n})$ and its dual $\BMO(\bR^{n})$. It is then proven in
\cite{BIJZ} that
\begin{equation}\label{nocancel}
f g\in L^{1}(\bR^{n})+\H^{\Phi}_\omega(\bR^{n}).
\end{equation}
Here $\mathcal H^\Phi_\omega(\bR^n)$ is the weighted Hardy-Orlicz space related to the
Orlicz function
\begin{equation}\label{orl-log}
\Phi(t):=\frac t{\log (e+t)}
\end{equation}
and with weight $\omega(x):=(\log (e+|x|))^{-1}$.

Our aim is to improve this result in many directions. The first one consists in proving that the space $\H^{\Phi}_\omega(\bR^{n})$ can be replaced by a smaller space. More precisely, we define the Musielak-Orlicz space $L^{\log}(\bR^{n})$ as the space of measurable functions $f$ such that
$$\int_{\mathbb R^n} \frac {|f(x)|}{ \log(e+|x|) +\log (e+ |f(x)|)}dx <\infty.$$
The space $\H^{\log}(\bR^n)$ is then defined, as usual, as the space of tempered distributions for which the grand maximal function  is in $L^{\log}(\bR^{n})$. This is a particular case of a Hardy space of Musielak-Orlicz type, with a variable (in $x$) Orlicz function that is also called a Musielak-Orlicz function (see \cite{Ky1}). This kind of space had not yet been considered. A systematic study of Hardy spaces of Musielak-Orlicz type has been done separately by the last author \cite{Ky1}. It generalizes the work of Janson \cite{Ja} on Hardy-Orlicz spaces. In particular, it is proven there that the dual of the space $\H^{\log}(\bR^n)$ is the  generalized $\BMO$ space  that has been introduced by Nakai and Yabuta (see \cite {NY}) to characterize  multipliers of $\BMO(\bR^n)$. Remark that by duality with our result, functions $f$ that are bounded and in the dual of $\H^{\log}(\bR^n)$ are multipliers of $\BMO(\bR^n)$. By the theorem of Nakai and Yabuta there are no other multipliers, which, in some sense, indicates that $\H^{\log}(\bR^n)$ could not be replaced by a smaller space.

Secondly we  answer   a question of \cite{BIJZ} by proving that there exists continuous bilinear operators that allow to split the product into an $L^1(\bR^n)$ part and a part in this Hardy Orlicz space $\H^{\log}(\bR^n)$. More precisely we have the following.

 \begin{theorem} \label{main}There exists two continuous bilinear operators on the product space $\H^1(\bR^n)\times \BMO(\bR^n)$, respectively   $S:\H^1(\bR^n)\times \BMO(\bR^n) \mapsto L^{1}(\bR^{n})$ and $T:\H^1(\bR^n)\times \BMO(\bR^n) \mapsto \H^{\log}(\bR^{n})$
 such that
 \begin{equation}
 fg=S(f,g)+T(f,g).
 \end{equation}
 \end{theorem}
 The operators $S$ and $T$ are defined in terms of a wavelet decomposition. The operator $T$ is defined in terms of paraproducts. There is no uniqueness, of course. In fact, the same decomposition of the product $fg$ has already been considered by Dobyinsky and Meyer (see \cite{DM, Do1, Do2}, and also \cite{CDM, Ch}). The action of replacing the product by the operator $T$ was called by them a {\sl renormalization} of the product. Namely, $T$ preserves the cancellation properties of the factor, while $S$ does not. Dobyinsky and Meyer considered $L^2$ data for both factors, and showed that $T(f,g)$ is in the Hardy space $\H^1(\bR^n)$. What is surprising in our context is that both terms inherit some properties of the factors. Even if the product $fg$ is not integrable, the function $S(f,g)$ is, while $T(f,g)$ inherits  cancellation properties of functions in Hardy spaces without being integrable. So, in some way  each term has more properties than expected at first glance.

 Another implicit conjecture of \cite{BIJZ} concerns bilinear operators with cancellations, such as the ones involved in the  $\div$-$\curl$ lemma for instance. In this case it is expected that there is no $L^1$ term. To illustrate this phenomenon, it has been proven in \cite{BFG} that, whenever $F$ and $G$ are two vector fields respectively in $\H^1(\bR^n,\bR^n )$ and  $\BMO(\bR^n,\bR^n )$ such that $F$ is  $\curl$-free and $G$ is $\div$-free, then their scalar product $F\cdot G$ is in $\H^\Phi_w(\bR^n,\bR^n )$ (in fact there is additional  assumption on the $\BMO$ factor). By using the same technique as Dobyinsky to deal with the terms coming from $S$, we give a new proof, without any additional assumption. Namely, we have the following.

 \begin{theorem} \label{div-curl} Let $F$ and $G$ be  two vector fields, one of them in  $\H^1(\bR^n,\bR^n )$ and the other one in  $\BMO(\bR^n,\bR^n )$, such that $\curl F=0$  and  $\div G=0$.Then their scalar product $F\cdot G$ (in the distribution sense) is in $\H^{\log}(\bR^n )$.
 \end{theorem}

 In Section 2 we introduce the spaces $L^{\log}(\bR^n )$ and $\H^{\log}(\bR^n )$ and give the generalized H\"{o}lder inequality that plays a central role when dealing with products of functions respectively in $L^1(\bR^n)$ and $\BMO(\bR^n)$.
 In Sections 3 and 4 we give prerequisites on wavelets and recall the $L^2$ estimates of Dobyinsky. We prove Theorem \ref{main} in Section 5 and Theorem \ref{div-curl} in Section 6.

\section{The space $\H^{\log}(\bR^n)$ and a generalized H\"{o}lder inequality}
We first define the (variable) Orlicz function
$$\theta(x, t):= \frac t {\log (e+|x|)+\log (e+t)}$$
for $x\in\bR^n$ and $t>0$. For fixed $x$ it is an increasing function while $t\mapsto \theta(x, t)/t$ decreases. We have $p<1$ in the following inequalities satisfied by $\theta$.
\begin{eqnarray}
\theta(x, st)&\leq C_p s^p \theta(x,t) \qquad \qquad &0<s<1 \label{less1}\\
\theta(x, st)&\leq s \theta(x,t) \qquad \qquad &s>1 \label{larger1}.
\end{eqnarray}
These two properties are among the ones that are usually required for (constant) Orlicz functions in Hardy Theory, see for instance \cite{Ja, BG, Ky1}. They guarantee, in particular, that $L^{\log}(\bR^n)$, defined as the set of functions $f$ such that
$$\int_{\bR^n} \theta(x, |f(x)|) dx <\infty$$
 is a  vector space. For $f\in L^{\log}(\bR^n)$, we define
$$\|f\|_{L^{\log}}:=\inf \{\lambda>0\;;\; \int_{\bR^n} \theta(x, |f(x)|/\lambda) dx \leq 1\}.$$
It is not a norm, since it is not sub-additive. In place of sub-additivity, there exists a constant $C$ such that, for $f, g\in  L^{\log}(\bR^n)$,
$$\|f+g\|_{L^{\log}}\leq C (\|f\|_{L^{\log}}+ \|g\|_{L^{\log}}).$$
On the other hand, it is homogeneous.

The space $L^{\log}(\bR^n)$ is a complete metric space, with the distance given by
$$\dist(f,g):=\inf \{\lambda>0\;;\; \int_{\bR^n} \theta(x, |f(x)-g(x)|/\lambda) dx \leq \lambda\}$$
(see \cite{RR}, from which proofs can be adapted, and \cite{Ky1}). Because of \eqref{less1}, a sequence $f_k$ tends to $0$ in $L^{\log}(\bR^n)$ for this distance if and only  if $\|f_k\|_{L^{\log}}$ tends to $0$.
\smallskip

Before stating our first proposition on products, we need some notations related to the space $\BMO(\bR^n)$. For  $Q$ a cube of $\bR^n$ and $f$ a locally integrable function, we note $f_Q$ the mean of $f$ on $Q$. We recall that a function $f$ is in $\BMO(\bR^n)$ if
$$\|f\|_{\BMO}:=\sup_Q \frac 1{|Q|}\int_Q |f-f_Q| dx <\infty.$$
We note $\mathbb{Q}:=[0,1)^n$ and, for $f$ a function in $\BMO(\bR^n)$,
$$\|f\|_{\BMO^+}:=|f_{\mathbb{Q}}|+\|f\|_{\BMO}.$$
This is a norm, while the $\BMO$ norm is only a norm on equivalent classes modulo constants.

The aim of this section is to prove the following proposition, which replaces H\"{o}lder Inequality in our context.
\begin{proposition} \label{product_sharp}Let $f\in L^1(\mathbb R^n)$ and $g\in\BMO(\bR^n)$. Then the product $fg$ is in $L^{\log}(\bR^n)$. Moreover, there exists some constant $C$ such that
$$\|f g\|_{L^{\log}}\leq C \|f\|_{L^1}\|g\|_{\BMO^+}.$$
 \end{proposition}
 \begin{proof} It is easy to adapt the proof given in \cite{BIJZ}, which leads to a weaker statement. We prefer to give a complete proof here, which has the advantage to be easier to follow than the one given in \cite{BIJZ}. We first  restrict to functions $f$ of norm $1$ and functions $g$ such that $g_{\mathbb Q}=0$ and $\|g\|_{\BMO}\leq \alpha$ for some uniform constant $\alpha$. Let us  prove in this case the existence of a uniform constant $\delta$ such that
\begin{equation}\label{equ1}
 \int_{\bR^n} \theta(x, |f(x)g(x)|) dx \leq \delta.
\end{equation}
 The constant $\alpha$ is chosen so that, by John-Nirenberg inequality, one has the inequality
 $$\int_{\R^n} \frac{e^{|g|}}{(e+|x|)^{n+1}} dx \leq \kappa, $$
 with $\kappa$ a uniform constant that depends only of the dimension $n$ (see \cite{St}).
  Our main tool is the following lemma.
 \begin{lemma} Let $M\geq 1$. The following inequality holds for $s, t>0$,
 \begin{equation}\label{hoder-variable}
    \frac {st}{M+\log(e+st)}\leq e^{t-M}+s.
 \end{equation}
 \end{lemma}
 \begin{proof} By monotonicity it is sufficient to consider the case when $s=e^{t-M}$. More precisely, it is sufficient to prove that
 $$\frac {t}{M+\log(e+te^{t-M})}\leq 1.$$
 This is direct when $t\leq M$. Now, for $t\geq M$, the denominator is bounded below by $M+t-M$, that is, by $t$.
 \end{proof}
Let us go back to the proof of the proposition. We choose $M:= (n+1)\log(e+|x|)$. Then
$$ \frac {|f(x)g(x)|}{(n+1)(\log (e+|x|) +\log (e+|f(x)g(x)|))} \leq  \frac {e^{|g(x)|}}{(e+|x|)^{n+1}} +|f(x)|.$$
After integration we get \eqref{equ1} with $\delta=(n+1)(\kappa +1)$. Let us then assume that $|g_{\mathbb Q}|\leq \alpha$ while the other assumptions on $f$ and $g$ are the same.
We then write $f g=f g_{\mathbb Q}+ f(g-g_{\mathbb Q})$ and find again the estimate \eqref{equ1} with $\delta=(n+1)(\kappa +1)+\alpha$. Using \eqref{less1}, this means that, for $\|f\|_{L^1}=1$ and $\|g\|_{\BMO^+}=\alpha$ and for $p<1$, we have the inequality $\|fg\|_{L^{\log}}\leq (\delta C_p)^{1/p}$. The general case follows by homogeneity, with $C=\delta \alpha^{-1}$.

\end{proof}
Remark that we only used the fact that $g$ is in the exponential class for the weight $(e+|x|)^{-(n+1)}$.
\smallskip

Finally let us define the  space $\H^{\log}(\bR^n)$. We first define the grand maximal function of a distribution $f\in \mathcal S' (\bR^n)$ as follows. Let $\mathcal \F$ be the set of functions $\Phi$ in $\S(\bR^n)$ such that $|\Phi(x)|+|\nabla \Phi(x)|\leq (1+|x|)^{-(n+1)}$. For $t>0$, let $\Phi_t(x):=t^{-n}\Phi(\frac xt)$. Then
\begin{equation}\mathcal M f(x):= \sup_{\Phi\in
\F}\sup_{t>0}|f*\Phi_t(x)|.\label{grand-maximal}
\end{equation}
By analogy with Hardy-Orlicz spaces, we define the space $\H^{\log}(\bR^n)$ as the space of tempered distributions such that $\mathcal M f$ in $L^{\log}(\bR^n)$. We need the fact that $\H^{\log}(\bR^n)$ is a complete metric space. Convergence in $\H^{\log}(\bR^n)$ implies convergence in distribution. The space $\H^1(\bR^n)$, that is, the space of functions $f\in L^1(\bR^n)$ such that $\mathcal M f$ in $L^{1}(\bR^n)$, is strictly contained in $\H^{\log}(\bR^n)$.

\section{Prerequisites on Wavelets}

Let us consider a wavelet basis of $\bR$ with compact support. More explicitly, we are first given a $\C^1(\bR)$ wavelet in Dimension one, called $\psi$, such that $\{2^{j/2}\psi (2^j x-k)\}_{j,k\in \bZ}$ form an $L^2(\bR)$ basis. We assume that this wavelet basis comes for a multiresolution analysis (MRA) on $\bR$, as defined below (see \cite{Me}).
\begin{Definition}
A multiresolution analysis (MRA) on $\mathbb R$ is defined as an increasing sequence $\{V_j\}_{j\in\bZ}$ of closed subspaces of $L^2(\mathbb R)$ with the following four properties

i) $\bigcap_{j\in\mathbb Z}V_j=\{0\}$ and $\overline{\bigcup_{j\in\mathbb Z}V_j}=L^2(\mathbb R)$,

ii) for every $f\in L^2(\mathbb R)$ and every $j\in\mathbb Z$, $f(x)\in V_j$ if and only if $f(2x)\in V_{j+1}$,

iii) for every $f\in L^2(\mathbb R)$ and every $k\in\mathbb Z$, $f(x)\in V_0$ if and only if $f(x-k)\in V_0$,

iv) there exists a function $\phi\in L^2(\mathbb R)$, called the scaling function, such that the family $\{\phi_k(x)=\phi(x-k): k\in\mathbb Z\}$ is an orthonormal basis for $V_0$.
\end{Definition}
It is classical that, when given an (MRA) on $\bR$, one can find a wavelet $\psi$ such that $\{2^{j/2}\psi (2^j x-k)\}_{k\in \bZ}$ is an orthonormal basis of $W_j$, the orthogonal complement of $V_j$ in $V_{j+1}$. Moreover, by Daubechies Theorem (see \cite{Da}), it is possible to find a suitable (MRA) so that $\phi$ and $\psi$ are $\C^1(\bR)$ and compactly supported, $\psi$ has mean $0$ and $\int x\psi (x)dx=0$, which is known as the moment condition. We could content ourselves, in the following theorems, to have $\phi$ and $\psi$ decreasing sufficiently rapidly  at $\infty$, but proofs are simpler with compactly supported wavelets. More precisely we assume that $\phi$ and $\psi$ are supported in the interval $1/2+m(-1/2, +1/2)$, which is obtained from $(0,1)$ by a dilation by $m$ centered at $1/2$.

Going back to $\bR^n$, we recall that a wavelet basis of $\bR^n$ is found  as follows. We call $E$ the set $E=\{0,1\}^n\setminus \{(0,\cdots, 0)\}$ and, for $\lambda \in E$, state $\psi^{\lambda}(x)=\phi^{\lambda_1}(x_1)\cdots \phi^{\lambda_n}(x_n)$, with
$\phi^{\lambda_j}(x_j)=\phi(x_j)$ for $\lambda_j =0$ while $\phi^{\lambda_j}(x_j)=\psi(x_j)$ for $\lambda_j =1$. Then the set $\{2^{nj/2}\psi^{\lambda} (2^j x-k)\}_{j\in \bZ, k\in\bZ^n, \lambda\in E}$ is an orthonormal basis of $L^2(\bR^n)$.
As it is classical, for $I$ a dyadic cube of $\bR^n$, which may be written as the set of $x$ such that $2^j x-k \in  (0,1)^n$, we note
$$\psi_I^{\lambda}(x)=2^{nj/2}\psi^{\lambda} (2^j x-k).$$
We also note $\phi_I=2^{nj/2}\phi_{(0,1)^n} (2^j x-k)$, with $\phi_{(0,1)^n}$ the scaling function in $n$ variables, given by
$\phi_{(0,1)^n}(x)=\phi(x_1)\cdots \phi(x_n)$.
In the sequel, the letter $I$ always refers to dyadic cubes. Moreover, we note $kI$  the cube of same center dilated by the coefficient $k$. Because of the assumption on the supports of $\phi$ and $\psi$, the functions
$\psi_I^{\lambda}$ and $\phi_I$ are supported in the cube $mI$.

The wavelet basis $\{\psi_I^\lambda\}$, obtained by letting $I$ vary among dyadic cubes and $\lambda$ in $E$, comes from an (MRA) in $\bR^n$, which we still note $\{V_j\}_{j\in\bZ}$, obtained by taking tensor products of the one dimensional ones. The functions $\phi_I$, taken for a fixed length $|I|=2^{-jn}$, form a basis of $V_j$.  As in the one dimensional case we note $W_j$ the orthogonal complement of $V_j$ in $V_{j+1}$. As it is classical, we note $P_j $ the orthogonal projection onto $V_j$ and $Q_j$ the orthogonal projection onto $W_j$. In particular,
\begin{eqnarray*}
  f &=& \sum_{i\in \bZ} Q_i f  \\
   &=& P_j f+ \sum_{i\geq j} Q_i f.
\end{eqnarray*}

\section{The $L^2$ estimates for the product of two functions}

We summarize here the main results of Dobyinsky \cite{Do2}.

Let us consider two $L^2$ functions $f$ and $g$, which we express through their wavelet expansions, for instance
$$f=\sum_{\lambda \in E} \sum_I \langle f, \psi_I^{\lambda}\rangle \psi_I^{\lambda}.$$ Then, when $f$ and $g$ have a finite wavelet expansion, we have
\begin{eqnarray}
fg
&=&
\sum_{j\in\mathbb Z}(P_jf)(Q_j g)+\sum_{j\in\mathbb Z}(Q_j f)(P_j g)+\sum_{j\in\mathbb Z}(Q_j f)(Q_j g)\label{decomp}\\
&:=&
\Pi_1 (f,g)+\Pi_2(f,g)+\Pi_3(f,g)\nonumber.
\end{eqnarray}
The two operators $\Pi_1$ and $\Pi_2$ are called paraproducts. A posteriori each term of  Formula \eqref{decomp}  can be given a meaning for all functions $f,g\in L^2(\bR^n)$. Indeed the two operators $\Pi_1$ and $\Pi_2$, which coincide, up to permutation of $f$ and $g$, extend as bilinear operators from $L^2(\bR^n)\times L^2(\bR^n)$ to $\H^1(\bR^n)$, see \cite{Do2}, while the operator $\Pi_3$ extends to an operator from $L^2(\bR^n)\times L^2(\bR^n)$ to $L^1(\bR^n)$.

The two $L^2$ estimates are given in the following two lemmas. We sketch their proof for the convenience of the reader as this will be the basis of our proofs in the context of $\H^1(\bR^n)$ and $\BMO(\bR^n)$. Details may be found in  \cite{Do2}.
\begin{lemma} \label{pi3L2}
The bilinear operator $\Pi_3$  is a bounded operator from $L^2(\bR^n)\times L^2(\bR^n)$ into $L^1(\bR^n)$.
\end{lemma}
\begin{proof}
The series $\sum_{j\in\bZ} Q_jf Q_j g$ is normally convergent in $L^1(\bR^n)$, with
\begin{eqnarray*}
\sum _{j\in\bZ} \|Q_jf Q_jg\|_{L^1}&\leq&  \sum_{j\in\bZ} \|Q_jf \|_{L^2}\| Q_jg\|_{L^2}\\
&\leq &\Big(\sum_{j\in\bZ} \|Q_jf \|_{L^2}^2\Big)^{1/2}\Big(\sum_{j\in\bZ} \|Q_jg \|_{L^2}^2\Big)^{1/2}\\
 &\leq& C \|f\|_{L^2} \|g\|_{L^2}.
\end{eqnarray*}
This concludes for $\Pi_3$.
\end{proof}
\begin{lemma}\label{pi1L2}
The bilinear operator $\Pi_1$, a priori well defined for $f$ and $g$ having  a finite wavelet expansion, extends to $L^2(\bR^n)\times L^2(\bR^n)$ into a bounded operator to $\H^1(\bR^n)$.
\end{lemma}
\begin{proof}
Let us recall that one can write
 $$P_j f =  \sum_{|I|=2^{-jn}} \langle f, \phi_I\rangle \phi_I.$$
 This means that  $P_jf Q_j g$ can be written as a linear combination of $\psi_I^{\lambda}\phi_{I'}$, with $|I|=|I'|=2^{-jn}$. As before, for fixed $I$, this function is non zero only for a finite number of $I'$. More precisely,  such $I'$s can be written as $k2^{-j}+I$, with $k\in K$, where $K$ is the set of points with integer coordinates contained in $(-m, +m]^n$.   So $\Pi_1(f,g)$ can be written as a sum in $\lambda\in E$ and $k\in K$ of
 $$F_{\lambda, k}:=\sum_{j\in\bZ}\sum_{|I|=2^{-jn}}\langle f, \phi_{k2^{-j}+I} \rangle\langle g, \psi_I^\lambda \rangle \phi_{k2^{-j}+I} \psi_I^{\lambda}.$$
  At this point, we use the fact that the functions $|I|^{1/2}\phi_{k2^{-j}+I} \psi_I^{\lambda}$ are of mean zero because of the orthogonality of $V_j$ and $W_j$. Moreover they are of class $\C^1(\bR^n)$ and are obtained from the one for which $I=(0,1)^n$ through the same process of dilation and translation as the wavelets. So they form what is called a system of molecules. It is well-known (see Meyer's book \cite{Me}) that such a linear combination of molecules has its  $\H^1$ norm bounded by $C$ times the $\H^1$ norm of the  linear combination of wavelets with the same coefficients. Namely, we are linked to prove that
$$\|\sum_j\sum_{|I|=2^{-jn}}\sum_{\lambda \in E}\langle f, \phi_{k2^{-j}+I} \rangle\langle g, \psi_I^\lambda \rangle 2^{nj/2} \psi_I^{\lambda}\|_{\H^1}\leq C\|f\|_{L^2} \, \|g\|_{L^2}.$$
We use the characterization of $\H^1(\bR^n)$ through wavelets to bound this norm by the $L^1$ norm of its square function, given by
$$\left(\sum_j\sum_{|I|=2^{-jn}}\sum_{\lambda \in E}|\langle f, \phi_{k2^{-j}+I} \rangle\langle g, \psi_I^\lambda \rangle |^2 2^{nj}|I|^{-1} \chi_I\right)^{1/2}.$$
This function is bounded at $x$ by
$$\sup_{I\ni x} |\langle f, |I|^{-1/2}\phi_I \rangle| \times \left(\sum_j\sum_{|I|=2^{-jn}}\sum_{\lambda \in E}|\langle g, \psi_I^\lambda \rangle |^2 |I|^{-1} \chi_I(x)\right)^{1/2}.$$
The first factor is bounded, up to a constant, by the Hardy Littlewood maximal function of $f$, which we note $Mf$. We conclude by using Schwarz inequality, then the maximal theorem to bound the $L^2$ norm of $Mf$ by the $L^2$ norm of $f$, then the fact that the $L^2$ norm of the second factor is the $L^2$ norm of $g$.
\end{proof}

We will need the expression of $\Pi_1(f, g)$ and $\Pi_2(f, g)$ when $f$ has a finite wavelet expansion while $g$ in only assumed to be in $L^2(\bR^n)$. The following lemma is immediate for $g$ with a finite wavelet expansion, then by passing to the limit otherwise.
\begin{lemma} Assume that $f$ has a finite wavelet expansion and $Q_j f=0$ for $j\notin [j_0, j_1)$. For $g\in L^2(\bR^n)$, one has
\begin{eqnarray}
 \Pi_1(f, g) &=& \sum_{j= j_0}^{j_1-1} P_j f Q_j g + f\sum_{j\geq j_1} Q_j g \label{pile1expression}\\
 \Pi_2(f, g) &=& f P_{j_0}g +  \sum_{j=j_0}^{j_1-1} Q_j f \left(\sum_{j_0\leq i\leq j-1} Q_i g \right). \label{pile2expression}
 \end{eqnarray}
 \end{lemma}
 \section{Products of functions in $\H^1(\bR^n)$ and $\BMO(\bR^n)$}

Let us first recall the wavelet characterization of $\BMO(\bR^n)$: if $g$ is in $\BMO(\bR^n)$, then for all (not necessarily dyadic) cubes $R$ , we have that
$$\Big(|R|^{-1}\sum_{\lambda\in E}\sum_{I\subset R}|\langle g,\psi_I^{\lambda}\rangle|^2\Big)^{1/2} \leq C\|g\|_{\BMO},$$
 and the supremum over all cubes $R$ of the left hand side is equivalent to the $\BMO$ norm of $g$.

Remark that the wavelet coefficients of a function $g$ in $\BMO$ are well defined since $g$ is locally square integrable. The $\langle g,\phi_I\rangle$'s are well defined as well. So $Q_j g$ makes sense, as well as $P_j g$. Indeed, they are sums of the corresponding series in $\psi_I^{\lambda}$ or $\phi_I$ with $|I|=2^{-jn}$, and at each point only a finite number of terms are non zero.

Moreover, we claim that \eqref{pile1expression} and \eqref{pile2expression} are well defined for $f$ with a finite wavelet expansion and $g$ in $\BMO(\bR^n)$. This is direct for $\Pi_2(f,g)$. For $\Pi_1(f,g)$, it is sufficient to see that
the series $\sum_{j\geq j_1} Q_j g$ converges in $L^2(R)$, where $R$ is a large cube containing the support of $f$. This comes from the wavelet characterization of $\BMO(\bR^n)$. Indeed, on $R$ one has
$$\sum_{ j_1\leq j\leq k} Q_j g=\sum_{\lambda\in E}\sum_{I\subset mR, 2^{-nk}\leq |I|\leq 2^{-nj_1}}\langle g, \psi_I^{\lambda}\rangle \psi_I^{\lambda}. $$
This is the partial sum of an orthogonal series, that converges in $L^2(\bR^n)$.

As a final remark, we find the same expressions for $\Pi_1(f, g)$, $\Pi_2(f, g)$, $\Pi_3(f, g)$ and $fg$ when $g$ is replaced by $\eta g$, where $\eta$ is a smooth compactly supported function such that $\eta$ is equal to $1$ on a large cube $R$. Just take $R$ sufficiently large to contain the supports of $f$, $Q_jf$, and  all   functions $\phi_I$ and $\psi_I^{\lambda}$ that lead to a non zero contribution in the expressions of the four functions under consideration. Since $\eta g$ is in $L^2(\bR^n)$, we have the identity \eqref{decomp}. This leads to the identity
\begin{equation}\label{decomp2}
    fg=\Pi_1(f, g)+\Pi_2(f, g)+\Pi_3(f, g).
\end{equation}
So Theorem \ref{main} will be a consequence of the boundedness of the operators $\Pi_1(f, g)$, $\Pi_2(f, g)$ and $\Pi_3(f, g)$.

\smallskip

Before considering this boundedness, we describe the atomic decomposition of the Hardy space $\H^1(\bR^n)$, which will play a fundamental role in the proofs.

We recall that a function $a$ is called a (classical) atom of $\H^1(\bR^n)$ related to the (not necessarily dyadic) cube $R$ if $a$ is in $L^2(\bR^n)$, is supported in $R$, has mean zero and is such that $\|a\|_{L^2}\leq |R|^{-1/2}$.

For simplicity we will
consider atoms that are adapted to the wavelet basis under consideration. More precisely, we  call the function $a$ a $\psi$-atom related to the dyadic cube $Q$ if it is an $L^2$ function that may be written as
\begin{equation}\label{atom}
a=\sum_{I\subset R}\sum_{\lambda \in E} a_{I, \lambda} \psi_I^\lambda
\end{equation}
such that, moreover, $\|a\|_{L^2}\leq|R|^{-1/2}$.
Remark that $a$ is compactly supported in  $mR$ and has mean $0$, so that it is a classical atom related to $mR$, up to the multiplicative constant $m^{n/2}$. It is standard that an atom is in $\H^1(\bR^n)$ with norm bounded by a uniform constant. The atomic decomposition gives the converse.
\begin{Theorem}[Atomic decomposition]
There exists some constant $C$ such that all functions $f \in \H^1(\bR^n)$ can be written as the limit in the distribution sense and in $\H^1$ of an infinite sum
\begin{equation}\label{atomic}
f =\sum_\ell \mu_{\ell} a_{\ell}
\end{equation}
with $a_{\ell}$ $\psi$-atoms related to some dyadic cubes $R_{\ell}$ and $\mu_{\ell}$ constants such that
$$\sum_{\ell} |\mu_{\ell}|\leq C \|f\|_{\H^1}.$$
Moreover, for $f$ with a finite wavelet series, we can choose an atomic decomposition with a finite number of  atoms  $a_{\ell}$, which have also a finite wavelet expansion extracted from the one of $f$.
\end{Theorem}
This theorem  is a small variation of a standard statement. The second part may be obtained easily by taking  the atomic decomposition given in \cite{HW}, Section 6.5. Remark that the interest of dealing with finite atomic decompositions has been underlined recently, for instance in \cite{MSV, MSV2}.

\smallskip

We want now to give sense to the decomposition \eqref{decomp} for $f\in \H^1(\bR^n)$ and $g\in \BMO(\bR^n)$. We will do it when $f$ has a finite wavelet expansion.

Let us first consider that two operators $\Pi_1$ and $\Pi_3$.

\begin{Theorem}
$\Pi_3$ extends into a bounded bilinear operator from $\H^1(\mathbb R^n)\times \BMO(\mathbb R^n)$ into $L^1(\mathbb R^n)$.
\label{pi3}\end{Theorem}
\begin{proof}We  consider $f$ with a finite wavelet expansion and $g\in \BMO(\bR^n)$, so that $\Pi_3(f, g)$ is well defined as a finite sum in $j$. Let us give an estimate of its $L^1$-norm.
We use the atomic decomposition of $f$ given in \eqref{atomic}, that is, $f =\sum_{\ell=1}^{L} \mu_{\ell} a_{\ell}$ where each $a_{\ell}$ is a $\psi$-atom related to the dyadic cube $R_{\ell}$ and $\sum_{\ell=1}^{L} |\mu_{\ell}| \leq C \|f\|_{\H^1}$. Recall that each atom has also a finite wavelet expansion extracted from the one of $f$. From this, it is sufficient to prove that, for a $\psi$-atom $a$, which is supported in $R$ and has $L^2$ norm  bounded by $|R|^{-1/2}$, we have the estimate
\begin{equation}\label{Pi3atom}
\|\Pi_3(a, g)\|_{L^1}\leq C\|g\|_{\BMO}.
\end{equation}
We claim that
$\Pi_3(a,g)=\Pi_3(a, b)$, where $b:=\sum_{\lambda\in E}\sum_{I\in 2mR}\langle g, \psi_I ^{\lambda}\rangle \psi_I^{\lambda}$. Indeed, in the wavelet expansion of $g$ we only have to consider at each scale $j$  the terms $\psi_I^{\lambda}$ for which $\psi_I^{\lambda} \psi_{I'}^{\lambda'}$ is not identically $0$ for all $I'$ contained in $R$ such that $|I|=|I'|=2^{-jn}$. In other words we want $mI\cap mI'\neq \emptyset$, which is only possible for $I$ in $2mR$. Now let us recall the wavelet characterization of $\BMO(\bR^n)$: for all cubes $Q$, we have that
$$\Big(|Q|^{-1}\sum_{\lambda\in E}\sum_{I\subset Q}|\langle g,\psi_I^{\lambda}\rangle|^2\Big)^{1/2} \leq C\|g\|_{\BMO},$$
 and the supremum on all cubes $Q$ of the left hand side is equivalent to the $\BMO$ norm of $g$. It follows that the $L^2$ norm of $b$ is bounded by $Cm^{n/2}|R|^{1/2}\|g\|_{\BMO}$. This allows to conclude for the proof of \eqref{Pi3atom}, using Lemma \ref{pi3L2}.
\end{proof}

\bigskip

Next we look at $\Pi_1$.
\begin{Theorem}
$\Pi_1$ extends into  a bounded bilinear operator from $\H^1(\mathbb R^n)\times BMO(\mathbb R^n)$ into $\H^1(\mathbb R^n)$.
\end{Theorem}
\begin{proof}
Again, we consider  $\Pi_1(f, g)$ for $f$ with a finite wavelet expansion, so that it is well defined by \eqref{pile1expression}. As in the previous theorem we can consider separately each atom. So, as before,   let $a$ be such a $\psi$-atom. One can estimate $\Pi_1(a, g)$ as in the previous theorem.  We again claim that
$\Pi_1(f,g)=\Pi_1(f, b)$, where $b:=\sum_{\lambda\in E}\sum_{I\in 2mR}\langle g, \psi_I ^{\lambda}\rangle \psi_I^{\lambda}$. We then use Lemma \ref{pi1L2} to conclude that
\begin{equation}\label{Pi2atom}
\|\Pi_1(a, g)\|_{\H^1}\leq C\|g\|_{\BMO},
\end{equation}
 which we wanted to prove.
\end{proof}

We now consider the last term.
\begin{Theorem}\label{pi2}
$\Pi_2$ extends into  a bounded bilinear operator from $\H^1(\mathbb R^n)\times BMO^+(\mathbb R^n)$ into $\H^{\log}(\bR^n)$.
\end{Theorem}
\begin{proof}
The main point is the following lemma.
\begin{lemma}\label{writing}
let $a$ be a $\psi$-atom with a finite wavelet expansion related to the cube $R$ and $g\in \BMO$. Then we can write
\begin{equation}\label{atom-maj}
    \Pi_2(a,g)= h^{(1)}+ \kappa g_R h^{(2)}
\end{equation}
where $\|h^{(1)}\|_{\H^1}\leq C   \|g\|_{\BMO}$ and $h^{(2)}$ is an atom related to $mR$. Here $g_R$ is the mean of $g$ on $R$ and $\kappa$ a uniform constant, independent of $a$ and $g$.
\end{lemma}
Let us conclude from the lemma, which we take for granted for the moment. Let $f= \sum_{\ell=1}^L \mu_{\ell} a_{\ell}$ be the atomic decomposition of the function $f$,  which has a finite wavelet expansion.  Let us prove the existence of some uniform constant $C$ such that
\begin{equation}\label{cauchy}
   \left\| \mathcal M \left(\sum_{\ell= 1}^{L}\mu_{\ell}\Pi_2(a_{\ell}, g)\right)\right\|_{L^{\log}}\leq C\|g\|_{\BMO^+}\left(\sum_{\ell=1}^{L}|\mu_{\ell}|\right).
\end{equation}
With obvious notations, we conclude directly for terms $h^{(1)}_{\ell}$, using the fact that $L^1(\bR^n)$ is contained in $L^{\log}(\bR^n)$. So it is sufficient to prove that
$$\left\| \mathcal M \left(\sum_{\ell= 1}^{L}\mu_{\ell}g_{R_\ell} h^{(2)}_{\ell}\right)\right\|_{L^{\log}}\leq C\|g\|_{\BMO^+}\left(\sum_{\ell= 1}^{L}|\mu_{\ell}|\right).$$
At this point we proceed as in \cite{BIJZ}. We use the inequality
$$\mathcal M \left(\sum_{\ell=1}^{L}\mu_{\ell}g_{R_\ell} h^{(2)}_{\ell}\right)\leq \sum_{\ell=1}^{L}|\mu_{\ell}| |g_{R_\ell}| \mathcal M \left( h^{(2)}_{\ell}\right).$$
Then we write $g_{R_\ell}=g+(g_{R_\ell}-g)$. For the first term, that is,
$$|g|\left(\sum_{\ell=1}^{L}|\mu_{\ell}|  \mathcal M \left( h^{(2)}_{\ell}\right)\right),$$
we use the generalized H\"{o}lder inequality given in Proposition \ref{product_sharp}. Indeed, $g$ is in $\BMO(\bR^n)$ and
the function $\mathcal M (a)$, for $a$ an atom, is uniformly in $L^1$, so that $\sum_{\ell=1}^{L}|\mu_{\ell}|  \mathcal M \left( h^{(2)}_{\ell}\right)$ has norm in $L^1$ bounded by $C \sum_{\ell=1}^{L}|\mu_{\ell}|$.
To conclude for \eqref{cauchy}, it is sufficient to prove that
$$\left\|\sum_{\ell=1}^{L}|\mu_{\ell}||g-g_{R_\ell}| \mathcal M \left(h^{(2)}_{\ell}\right)\right\|_{L^1}\leq C \sum_{\ell=1}^{L}|\mu_{\ell}|. $$
This is a consequence of the following uniform inequality, valid for $g\in\BMO(\bR^n)$ and $a$ an atom adapted to the cube $R$:
$$\int_{\bR^n} |g-g_{R}| \mathcal M \left(a\right)dx\leq C\|g\|_{\BMO}.$$
To prove this inequality, by using invariance through dilation and translation, we may assume that $R$ is the cube $\mathbb Q$. We conclude by using the following classical lemma.
\begin{lemma}
Let $a$ be a classical atom related to the cube $\mathbb Q$ and $g$ be in $\BMO(\bR^n)$. Then
$$\int_{\bR^n} |g-g_{\mathbb Q}| \mathcal M \left(a\right)dx\leq C\|g\|_{\BMO}.$$
\end{lemma}
\begin{proof}
We cut the integral into two parts. By Schwarz Inequality and the boundedness of the operator $\mathcal M$ on $L^2(\bR^n)$, we have
\begin{eqnarray*}
\int_{|x|\leq 2} |g-g_{\mathbb Q}| \mathcal M \left(a\right)dx &\leq& C\left(\int_{2\mathbb Q} |g-g_{\mathbb Q}|^2dx \right)^{1/2}\|a\|_{L^2}\\
&\leq&
C \|g\|_{\BMO},
\end{eqnarray*}
here one used $|g_{2\mathbb Q}-g_{\mathbb Q}|\leq C\|g\|_{\BMO}$. Next, for $|x|>2$ we have the inequality
$$\mathcal M \left(a\right)(x)\leq \frac C{(1+|x|)^{n+1}},$$
and the classical inequality (see Stein's book \cite {St})
$$\int_{\bR^n} \frac{|g-g_{\mathbb Q}|}{(1+|x|)^{n+1}}dx\leq C\|g\|_{\BMO}.$$
We have proven \eqref{cauchy}.
\end{proof}

It remains to prove Lemma \ref{writing}, which we do now.

\begin{proof}[Proof of Lemma \ref{writing}]
Let $a$  be a $\psi$-atom which is related to the dyadic cube $R$. Let $j_0$ be such that $|R|=2^{-nj_0}$. We assume that $a$ has a finite wavelet expansion, so that $\Pi_2(a, g)$ is given by  \eqref{pile1expression} for some $j_1>j_0$.
As before, we can write
$\Pi_2(a, g) = a P_{j_0} g + \Pi_2(a, b)$, where $b$ is defined  by
$b:=\sum_{\lambda\in E}\sum_{I\in 2mR}\langle g, \psi_I ^{\lambda}\rangle \psi_I^{\lambda}$. It follows again from the characterization of $\BMO$ function through wavelets that the $L^2$ norm of $b$ is bounded by  $C\|g\|_{\BMO} |R|^{1/2}$. We use the $L^2$ estimate given by Lemma \ref{pi1L2} to bound uniformly the $\H^1$ norm of $\Pi_2(a,b)$. This term goes into  $h^{(1)}$.

 It remains to consider  $a P_{j_0} g$. By definition of $P_{j_0} g$, it can be written as $a \sum_I \langle g, \phi_I\rangle \phi_I$, where the sum in $I$ is extended to all dyadic cubes such that $|I|= 2^{-nj_0}$ and $mI\cap m R\neq \emptyset$. There are at most $(2m)^n$ such terms in this sum, and it is sufficient to prove that each of them can be written as $ h_1+\kappa |g_R| h_2$, with $h_2$ a classical atom related to $mQ$ and $h_1$ such that $\|h_1\|_{\H^1}\leq C   \|g\|_{\BMO}$. Let us first remark that for each of these $(2m)^n$ terms, the function $h:=|I|^{1/2} \phi_I a$ is (up to some uniform constant) a classical atom  related to $mR$: indeed, it has mean value $0$ because of the orthogonality of $\phi_I$ and $\psi_{I'}$ when $|I'|\leq |I|$ and the norm estimate follows at once. In order to conclude, it is sufficient to prove that $h_1=(g_R-|I|^{-1/2}\langle g, \phi_I\rangle) h$ has the required property. We conclude easily by showing that $g_R-|I|^{-1/2}\langle g, \phi_I\rangle$ is  bounded by $C\| g\|_{\BMO}$. But this difference may be written as $\langle \gamma, g\rangle$, where $\gamma:=|R|^{-1}\chi_R- |I|^{-1/2} \phi_I$. The function $\gamma$ has zero mean, is supported in $2m R$ and has $L^2$ norm bounded by $2|R|^{-1/2}$. Thus, up to multiplication by some uniform constant, it is a classical atom related to the cube $2m R$. It has an $\H^1$ norm that is uniformly bounded and its scalar product with $g$ is bounded by the $\BMO$ norm of $g$, up to a constant, as a consequence of the $\H^1-\BMO$ duality.

 This concludes for the proof.

\end{proof}
We have finished the proof of Theorem \ref{pi2}, and also of the one of Theorem \ref{main}. Just take $S=\Pi_3$.

\end{proof}
\section{Div-Curl Lemma}
The aim of this section is to prove Theorem \ref{div-curl}. The methods that we develop  are inspired by  the papers of Dobyinsky in the case of $L^2(\bR^n)$. They are generalized in a forthcoming paper of the last author \cite{Ky2}.

Let us first make some remarks. By using the decomposition of each product $F_j G_j$ into $S(F_j, G_j)+ T(F_j, G_j)$, we already know that all terms $T(F_j, G_j)$ are in $\H^{\log}(\bR^n)$. So we claim that it is sufficient to prove that $\sum_{j=1}^n S(F_j, G_j)$ is also in $\H^{\log}(\bR^n)$. We first assume that $F$ is in $\H^1(\bR^n, \bR^n)$ and $G$ in $\BMO(\bR^n, \bR^n)$. Since $F$ is $\curl$-free, we can assume that $F_j$ is a gradient, or, equivalently, $F_j=R_jf$, where $R_j$ is the $j$-th Riesz transform and $f=-\sum_{j=1}^n R_j(F_j)\in \H^1(\bR^n)$ since $\H^1(\bR^n)$ is invariant under Riesz transforms. Next, since $G$ is $\div$-free, we have the identity $\sum_{j=1}^n R_j G_j=0$. So it is sufficient to prove that $S(R_j f, G_j)+S(f, R_jG_j)$ is in $\H^{\log}(\bR^n)$ for each $j$. So Theorem \ref{div-curl} is a corollary of the following proposition.
\begin{proposition} Let $A$ be an odd Calder\'on-Zygmund operator. Then, the bilinear operator $ S(Af, g)+S(f,A g)$ maps continuously $\H^1(\bR^n)\times \BMO(\bR^n)$ into $\mathcal H^{1}(\mathbb R^n)$.
\end{proposition}
\begin{proof}
We make a first reduction, which is done by Dobyinsky in \cite{Do2}. When considering $S(f, g)$ on $\H^1(\bR^n)\times \BMO(\bR^n)$, we can write it as
$S(f,g)=h+S_0(f, g)$ with $h\in\H^1(\bR^n)$,
where
\begin{equation}\label{principal}
 S_0(f, g)=   \sum_{\lambda\in E}\sum_I \langle f, \psi_I^{\lambda}\rangle  \langle g, \psi_I^{\lambda}\rangle |\psi_I^\lambda|^2.
\end{equation}
Indeed, $S(f,g)-S_0(f,g)$ may be written in terms of products $\psi_I^{\lambda}\psi_{I'}^{\lambda'}$, with $|I|=|I'|$, $(I,\lambda)\ne (I',\lambda')$.  These functions are of mean $0$ because of the orthogonality of the wavelet basis, have $L^2$ norm  bounded, up to a constant, by $|I|^{-1/2}$, and are supported in $mI$. So they are $C$ times atoms of $\H^1(\bR^n)$. Recall that they are non zero only if $I'= k|I|^{1/n}+I$, with $k\in K$, where $K$ is the set of points with integer coordinates contained in $(-m, +m]^n$. So, to prove that $S(f,g)-S_0(f, g)$ is in $\H^1(\bR^n)$ it is sufficient to use the fact that, for fixed $\lambda, \lambda'$ and $k$,
$$\sum_{I} |\langle f, \psi_I^{\lambda}\rangle |\,| \langle g, \psi_{k|I|^{1/n}+I}^{\lambda'}\rangle |\leq C \|f\|_{\H^1}\|g\|_{\BMO}.$$
This is a consequence of the wavelet characterization of $f$ in $\H^1(\bR^n)$ and $g$ in $\BMO(\bR^n)$ and the following lemma, which may be found in \cite{FJ}.
\begin{lemma}
There exists a uniform constant $C$, such that, for  $(a_I)_{I\in\mathcal D}$ and $(b_I)_{I\in\mathcal D}$ two sequences that are indexed by the set $\mathcal D$ of dyadic cubes ,
one has the inequality
$$\sum_{I\in\mathcal D}|a_I||b_I|\leq C \left\|\left(\sum_{I\in\mathcal D} |a_I|^2 |I|^{-1}\chi_I\right)^{1/2}\right\|_{L^1} \times\sup_{R\in\mathcal D} \left(|R|^{-1}\sum_{I\subset R}|b_I|^2\right)^{1/2}.$$
\end{lemma}

Let us come back to the proof of the proposition. From this first step, we conclude that it is sufficient to prove that  $B(f, g):=S_0(Af, g)+S_0(f, Ag)$ is in $\H^1(\bR^n)$. Using bilinearity as well as the fact that $A^*=-A$, we have
$$B(f,g):=\sum_{\lambda\in E}\sum_{\lambda'\in E}\sum_{I,I'}\langle f, \psi_I^{\lambda}\rangle  \langle g, \psi_{I'}^{\lambda'}\rangle \langle A\psi_I^{\lambda}, \psi_{I'}^{\lambda'}\rangle (|\psi_{I'}^{\lambda'}|^2-|\psi_{I}^{\lambda}|^2).$$
From this point, the proof is standard.  An explicit computation gives that $|\psi_{I'}^{\lambda'}|^2-|\psi_{I}^{\lambda}|^2$ is in $\H^1(\bR^n)$, with
$$\||\psi_{I'}^{\lambda'}|^2-|\psi_{I}^{\lambda}|^2\|_{\H^1}\leq C\left(\log (2^{-j}+2^{-j'})^{-1}+\log(|x_I-x_{I'}|+2^{-j}+2^{-j'})\right).$$
Here $|I|=2^{-jn}$ and $|I'|=2^{-j'n}$, while $x_I$ and $x_{I'}$ denote the centers of the two cubes.  Next we use the well-known estimate of the matrix of a Calder\'on-Zygmund operator (see [MC, Proposition 1]): there exists some
$\delta\in (0,1]$, such that
$$|\langle A\psi_I^{\lambda}, \psi_{I'}^{\lambda'}\rangle|\leq C p_{\delta}(I, I')$$
with $$ p_{\delta}(I, I') = 2^{-|j-j'|(\delta+n/2)}\Big(\frac{2^{-j}+2^{-j'}}{2^{-j}+2^{-j'}+|x_I-x_{I'}|}\Big)^{n+\delta}.$$
So, by using the inequality
$$\log\Big(\frac{2^{-j}+2^{-j'}+|x_I-x_{I'}|}{2^{-j}+2^{-j'}}\Big)\leq \frac{2}{\delta}\Big(\frac{2^{-j}+2^{-j'}+|x_I-x_{I'}|}{2^{-j}+2^{-j'}}\Big)^{\delta/2},$$
we obtain
$$\|B(f,g)\|_{\H^1}\leq C \sum_{I,I'}|\langle f, \psi_I^{\lambda}\rangle|\,|  \langle g, \psi_I^{\lambda'}\rangle|p_{\delta'}(I, I')$$
where $\delta'= \delta/2>0$.
We conclude by using the fact that the almost diagonal matrix $p_{\delta'}(I, I')$ defines a bounded operator on the space  of all sequences $(a_I)_{I\in \mathcal D}$ such that $\Big(\sum_I |a_I|^2 |I|^{-1}\chi_I\Big)^{1/2}\in L^1(\mathbb R^n)$.

This is the end of the proof of Theorem \ref{div-curl} for $F\in \H^1(\bR^n, \bR^n)$ and $G\in \BMO(\bR^n, \bR^n)$ with $\curl F=0$ and $\div G=0$. Assume now that
$\div F=0$  and  $\curl G=0$. Similarly as above, we have $\sum_{j=1}^n R_j F_j=0$ and $G_j=R_j g$ where $g=-\sum_{j=1}^n R_j G_j\in \BMO(\bR^n)$ since $\BMO(\bR^n)$ is invariant under Riesz transforms. Hence,
$$F\cdot G=\sum_{j=1}^n (T(F_j, G_j)+ S(F_j, G_j))=\sum_{j=1}^n T(F_j, G_j)+ \sum_{j=1}^n (S(F_j, R_j g)+ S(R_j F_j, g)).$$
We conclude as before from the proposition.
\end{proof}

\section{Acknowledgements}
The authors are partially supported by the project ANR AHPI number ANR-07-BLAN-0247-01.

\end{document}